\documentclass[12pt]{article}
\usepackage[utf8]{inputenc}
\usepackage{amssymb}
\usepackage{amsmath}
\usepackage{amsthm}
\usepackage{comment}
\usepackage{enumerate}

\title{Nonvanishing cohomology classes on finite groups of Lie type with Coxeter number at most $p$}
\date{Spring 2014}
\author{David Sprehn}

\theoremstyle{definition}
\newtheorem{thm}{Theorem}
\newtheorem{prop}[thm]{Proposition}
\newtheorem{lemma}[thm]{Lemma}
\newtheorem{cor}[thm]{Corollary}
\newtheorem{defn}[thm]{Definition}

\newcommand{\im} {\operatorname{Im}}
\renewcommand{\hom} {\operatorname{Hom}}
\newcommand{\aut} {\operatorname{Aut}}

\newcommand{\hgt} {\operatorname{ht}}

\newcommand{\iso} {\approx}

\newcommand{\normal} {\triangleleft}
\newcommand{\semidirect} {\rtimes}

\newcommand{\inv} {\ensuremath ^{-1}}

\newcommand{\bbZ} {\ensuremath{\mathbb{Z}}}

\newcommand{\F} {\ensuremath{\mathbb{F}}}

\newcommand{\set}[2] {\left\lbrace {#1} \,\middle\arrowvert\, {#2} \right\rbrace}

\begin{document}
\maketitle

\begin{abstract}
We prove that the degree $r(2p-3)$ cohomology of any finite group of Lie type over $\F_{p^r}$, with coefficients in characteristic $p$, is nonzero as long as its Coxeter number is at most $p$.  We do this by providing a simple explicit construction of a nonzero element.
\end{abstract}

\section{Introduction}
We investigate the cohomology of the finite groups of Lie type over $\F_{p^r}$ with coefficients in characteristic $p$.  The most important examples are the general linear groups over finite fields, $GL_n(\F_{p^r})$.  The cohomology rings of these groups in characteristics other than $p$ were determined by Quillen~\cite{QuillenK} in the course of his algebraic K-theory computations, but the cohomology in characteristic $p$ remains elusive.  When $n\leq p$, we prove that the lowest nontrivial cohomology groups occur in degree $r(2p-3)$.  This extends a result of Bendel, Nakano, and Pillen~\cite{BNP} (which was valid for $n\leq p-2$).   Furthermore, we do so by providing a simple explicit construction, valid in all (untwisted) groups of Lie type $G$ with Coxeter number at most $p$, of a nonzero class in that degree.

More specifically, we provide an embedding of the invariants
\[ H^*(\F_{p^r};\F_p)^{\F_{p^r}^\times} \iso H^*(GL_2(\F_{p^r});\F_p) \]
into $H^*(G;\F_p)$.  Since these invariants are nonzero~\cite[Lem.~A.1]{FP} in degree $r(2p-3)$, this shows $H^{r(2p-3)}(G;\F_p)\neq0$.  In particular,
\begin{thm}
Suppose $n\leq p$.  Then
\[ H^{r(2p-3)}(GL_n(\F_{p^r});\F_p)\neq0. \]
\end{thm}
In the case $r=1$, the above invariants contain nonvanishing classes in degrees $2p-3$ and $2p-2$, connected by the Bockstein homomorphism.  Hence, for all $n\leq p$, $H^*(GL_n(\F_p);\F_p)$ also contains such classes, answering a conjecture of Barbu~\cite{Barbu} who constructed a class in degree $2p-2$ and conjectured it to be a Bockstein.  We remark that this class in $H^{2p-2}(GL_n(\F_p);\F_p)$ may also be viewed as a Chern class of the permutation representation of $GL_n(\F_p)$ on the set $\F_p^n-\{0\}$.

For $G=GL_n(\F_{p^r})$, Friedlander and Parshall~\cite{FP} proved that the cohomology of $G$ vanishes below degree $r(2p-3)$, i.e.~the upper bound for vanishing given here is sharp.  It is not sharp in general for groups of Lie type.  However, Bendel, Nakano, and Pillen~\cite{BNP,BNP2} have shown that the bound $r(2p-3)$ is sharp in many cases, including the simply-laced, adjoint type groups of Coxeter number less than $p$, and the simply-connected groups of types $A_n$ ($n\geq 4$), $E_6$, $E_7$, and $E_8$, with sufficiently large $p$.  Furthermore, in simply-connected type $A_{n-1}$ with $5\leq n\leq (p-1)/2$ they showed the lowest cohomology group $H^{r(2p-3)}(SL_n(\F_{p^r});\F_p)$ to be one-dimensional, so our nonvanishing class is in fact a generator in these cases.  It follows that the same result holds for $GL_n(\F_{p^r})$.  Also in simply-connected type $C_n$ (i.e.~$Sp(2n,\F_q)$) with $n<p/2$, they showed that the lowest nonvanishing cohomology $H^{r(p-2)}(Sp(2n,\F_q);\F_p)$ is one-dimensional, so the nonvanishing class we shall construct there is again a generator.

The author would like to thank Steve Mitchell, his thesis advisor, for many wonderful discussions.

\section{Notation for finite groups of Lie type}
Fix $q=p^r$.  We investigate the cohomology of finite groups of Lie type over $\F_q$: those groups which arise as the set of fixed-points of a $q$-Frobenius map $F$ acting on a connected reductive algebraic group.

Accordingly, let $\bar G$ be a connected reductive algebraic group over $\bar\F_q$.  Let $F:\bar G\to\bar G$ be a {\bf $q$-Frobenius map}.  That is,\footnote{Other authors refer to these as ``standard’’ or ``untwisted’’ Frobenius maps.} a map obtained from the endomorphism $F_q((a_{ij}))=(a_{ij}^q)$ of $GL_N(\bar\F_q)$ via some closed embedding $\bar G\leq GL_N(\bar\F_q)$ whose image is preserved by $F_q$.  Any such map is surjective (since $\bar G$ is connected).  The (finite) fixed-point group $\bar G^F$ of a $q$-Frobenius map is called a {\bf finite group of Lie type} over $\F_q$.

For discussion of the following notions, see Carter~\cite{Carter}.
Fix a choice of maximal torus and Borel subgroup $\bar T\leq\bar B\leq\bar G$ which are preserved by $F$.  Let $\bar U$ be the unipotent radical of $\bar B$.
Write $G=\bar G^F$, and similarly $B=\bar B^F$, $T=\bar T^F$, $U=\bar U^F$.
We remark that $B=U\semidirect T$ and $N=N_{\bar G}(\bar T)^F$ form a split BN-pair of characteristic $p$ in $G$, and that $U\leq G$ is $p$-Sylow subgroup.

Let $W=N/T$ be the Weyl group of $G$, $S$ be its set of Coxeter generators, and $\Phi$ be the associated root system with base (determined by $B$) 
\[ \Delta=\set{\alpha_s}{s\in S} \]
and positive roots $\Phi^+$.  For $\alpha\in\Phi$, denote by $X_\alpha$ the corresponding root subgroup of $G$.  Thus a root subgroup $X_\alpha$ is contained in $U$ if and only if $\alpha$ is positive, and $U$ is generated by the positive root subgroups.  In particular, for $s\in S$, we shall abbreviate
\[ X_s=X_{\alpha_s}=U\cap s\inv w_0\inv Uw_0s \]
(where $w_0\in W$ is the longest element).  Each root subgroup is isomorphic to $\F_q$, and is normalized by $T$,
with the action being $\F_q$-linear.
Also define
\[ U_s=U\cap s\inv Us=\prod_{\alpha\in\Phi^+-\{\alpha_s\}}X_\alpha, \]
a normal subgroup of $U$;
the composite
\[ X_s\to U\to U/U_s \]
is an isomorphism.

Lastly, each $\alpha\in\Phi^+$ can be written
\[ \alpha=\sum_{s\in S} c_s\alpha_s \]
with the $c_i$ nonnegative integers; we define the height of $\alpha$ to be
\[ \hgt(\alpha) = \sum_{s\in S}c_s \in\bbZ^+ \]
and the Coxeter number of $G$ to be $h=\max\set{\hgt(\alpha)+1}{\alpha\in\Phi^+}$.

The reader may wish to keep in mind the case $G=GL_n(\F_{p^r})$, in which $n$ is the Coxeter number, $B$ (resp.~$U$) the group of invertible (resp.~unipotent) upper triangular matrices, $T$ the diagonal matrices,
$X_s$ the subgroups
\[ X_k=\set{(a_{ij})\in GL_n(\F_{p^r})}{a_{ij}=\delta_{ij}\mbox{ unless }(i,j)=(k,k+1)}, \]
and $U_s$ the subgroups
\[ U_k=\set{(a_{ij})\in U}{a_{k,k+1}=0}. \]

\section{Commuting regular unipotents}
Our goal is to provide an injection (Theorem~\ref{injection}) from the $T$-invariant cohomology of a root subgroup $X_s$ to the cohomology of $G$, by composing the pullback to $H^*(B;\F_p)$ with the transfer map up to $H^*(G;\F_p)$.  We will verify that the composite is injective by then restricting to a certain elementary abelian $p$-subgroup consisting of regular unipotent elements.  This section is devoted to showing the existence of such a subgroup (Corollary~\ref{ptorus}).

We remark that, in the case $G=GL_n(\F_{p^r})$, this section may be bypassed and Corollary~\ref{ptorus} proved by constructing the required subgroup directly as the (elementary abelian) subgroup generated by the matrices
\[ I+\lambda \begin{bmatrix}
0 & 1 & & \\ & \ddots & \ddots & \\ & & & 1 \\ & & & 0
\end{bmatrix}, \]
where $\lambda$ ranges over a choice of basis for $\F_{p^r}$ over $\F_p$.

\begin{defn}
For our purposes, an element $x\in G$ is unipotent if it lies in a conjugate of $U$, or equivalently, its order is a power of $p$.  Furthermore, $x$ is {\bf regular unipotent} in $G$ if its action on $G/B$ has a unique fixed point.  Since $N_G(B)=B$, this is equivalent to lying in a unique conjugate of $B$.
\end{defn}
There is of course a corresponding notion for an element of the algebraic group: a unipotent element $x\in\bar G$ is {\bf regular unipotent} in $\bar G$ if its action on $\bar G/\bar B$ has a unique fixed point, or equivalently, $x$ lies in a unique Borel subgroup.

For an element $x\in G$ of the finite group, then, there are two notions of being regular unipotent.  However, they coincide.  This follows from bijectivity~\cite[sec.~1.17]{Carter} of the map
\[ G/B=\bar G^F/\bar B^F\to (\bar G/\bar B)^F, \]
together with the observation that, when a regular unipotent in $\bar G$ is fixed by $F$, the (unique) Borel subgroup containing it is preserved by $F$.

The set of regular unipotent elements in $G$ is clearly preserved by conjugacy in $G$.  In fact, the above discussion shows that it is preserved by conjugacy in $\bar G$.

One can show (using the fact that each $U_s\normal U$):
\begin{lemma}  Let $x\in U$.
\begin{enumerate}[(a)]
\item $x$ is regular unipotent if and only if $x\notin U_s$ for all $s\in S$;
\item There exist regular unipotent elements in $G$.
\end{enumerate}
\end{lemma}
\noindent There are other ways to characterize the regular unipotents: see~\cite[Prop.~5.1.3]{Carter}.

We wish to show (Corollary~\ref{ptorus}) that, when the Coxeter number of $G$ is at most $p$, there exists an elementary abelian $p$-subgroup of rank $r$, in which every nontrivial element is regular unipotent.  (Recall $q=p^r$.)  This follows from two facts:
\begin{enumerate}
\item When the Coxeter number of $G$ is at most $p$, every nontrivial unipotent element has order $p$.
\item When $x\in G$ has order $p$, its $\bar G$-conjugacy class (plus the identity) contains an elementary abelian $p$-subgroup of rank $r$.
\end{enumerate}

The former statement is no doubt familiar, but we include an elementary proof in the appendix (Proposition~\ref{exponent}).  We shall show the latter using Testerman's Theorem on $A_1$-overgroups:

\begin{thm}[Testerman~\cite{Testerman}]
Let $\bar G$ be a semisimple algebraic group over an algebraically closed field $k$ of nonzero characteristic $p$.  Assume $p$ is a good prime\footnote{This means that $p$ does not divide the coefficients of any root $\alpha\in\Phi^+$ when expressed in the simple root basis.}
for $\bar G$.  Let $\sigma$ be a surjective endomorphism of $\bar G$ with finite fixed-point subgroup.  Let $u\in\bar G^\sigma$ with $u^p=1$.  Then there exists a closed connected subgroup $X$ of $\bar G$ with $\sigma(X)\leq X$, $u\in X$, and $X$ isomorphic to $SL_2(k)$ or $PSL_2(k)$.
\end{thm}
We begin by extending Testerman's theorem to reductive groups.
\begin{thm}\label{Testerman}
Let $\bar G$ be a connected reductive algebraic group over an algebraically closed field $k$ of nonzero characteristic $p$.  Assume $p$ is a good prime for $\bar G$.  Let $\sigma$ be a surjective endomorphism of $\bar G$ with finite fixed-point subgroup.  Let $u\in\bar G^\sigma$ with $u^p=1$.  Then there exists a closed connected subgroup $X$ of $\bar G$ with $\sigma(X)\leq X$, $u\in X$, and $X$ isomorphic to $SL_2(k)$ or $PSL_2(k)$.
\end{thm}

\begin{proof}
Since $\bar G$ is connected reductive, its derived subgroup $\bar G'$ is closed~\cite[sec.~2.3]{Borel}, connected and semisimple; it is preserved by $\sigma$.  We claim the restriction of $\sigma$ to $G’$ remains surjective.  This is because its image is a closed subgroup of $G’$ with the same dimension.  Furthermore, $G’$ contains all the unipotent elements of $\bar G$, including $u$; it also has the same root system.  Therefore, Testerman's theorem applies, producing the desired subgroup $X$.
\end{proof}

Using this, we will obtain:
\begin{prop}\label{conj}
Let $\bar G$ be a connected reductive group with a $p^r$-Frobenius map $F$.  Assume $p$ is a good prime for $\bar G$.  If $u\in\bar G^F$ has order $p$, then there exists an elementary abelian $p$-subgroup of rank $r$ in $\bar G^F$, containing $u$, all of whose nontrivial elements are conjugate in $\bar G$.
\end{prop}
For the finite groups of Lie type $\bar G^F$ with Coxeter number at most $p$, this proposition applies (by Proposition~\ref{exponent}) to the regular unipotent elements:
\begin{cor}\label{ptorus}
Let $G$ be a nontrivial\footnote{I.e.~positive rank.} group of Lie type over $\F_{p^r}$ with Coxeter number at most $p$.  There exists an elementary abelian $p$-subgroup of rank $r$ in $G$, all of whose nontrivial elements are regular unipotent.
\end{cor}
We point out that the assumption of $p$ being a good prime is automatic when $G$ has Coxeter number at most $p$.

In order to prove the proposition, we need the following lemma on Frobenius maps of a one-dimensional additive group.  It may be viewed as showing how to recover $q$ from a $q$-Frobenius map.

\newpage
\begin{lemma}\label{Ga}
Let $k$ be an algebraically closed field of characteristic $p$.\begin{enumerate}
\item Let $\phi:k[x]\to k[x]$ be determined by $\phi(x)=cx^{p^d}$ for some $c\in k^\times$ and $d\geq0$.  If $p(x)$ is any nonconstant polynomial such that $\phi(p(x))=p(x)^{p^e}$, then $d=e$.
\item Let $X\leq GL_N(k)$ be a (closed) subgroup isomorphic to $\mathbb{G}_a(k)$ which is preserved by the standard Frobenius map $F_{p^r}$.  Then (as finite groups)
\[ X^{F_{p^r}}\iso C_p^r. \]
\end{enumerate}
\end{lemma}
The first statement is easily checked, and the second follows after observing that any injective endomorphism of $\mathbb{G}_a$ must have the form $u\mapsto cu^{p^d}$ for some $c\in k^\times$ and $d\geq0$.

\begin{proof}[Proof of Proposition~\ref{conj}]
A $p^r$-Frobenius map $F$ is surjective with finitely many fixed points.  Then by Theorem~\ref{Testerman}, we know that $u$ is contained in a closed $F$-invariant subgroup isomorphic to either $SL_2(\bar\F_p)$ or $PSL_2(\bar\F_p)$.  In these groups, every element of order $p$ is conjugate.  Furthermore, the unipotent radical $U$ of an $F$-invariant Borel subgroup is isomorphic to $\mathbb{G}_a$, hence by Lemma~\ref{Ga} its $F$-invariants form an elementary abelian $p$-subgroup of rank $r$, all of whose nontrivial elements are conjugate in $\bar G$.
\end{proof}

\section{Cohomology}
Now we have all the ingredients in place for our theorem on cohomology of $G$.
Recall that $q=p^r$; $G,T,B,U,S$ are, respectively, a finite group of Lie type, a maximal torus, a Borel containing $T$, its unipotent radical, and the set of Coxeter generators for the Weyl group of $G$.  For each $s\in S$, $X_s$ is the corresponding simple root subgroup, which is isomorphic to $\F_q$.

\begin{thm}\label{injection}
Let $G$ be a finite group of Lie type over $\F_q$ having Coxeter number at most $p$.  For each $s\in S$, there is an injective graded vector space map from $(H^*(X_s;\F_p))^T$ to $H^*(G;\F_p)$.  It is given by composing the pullback $H^*(X_s\semidirect T;\F_p)\to H^*(B;\F_p)$ with transfer $H^*(B;\F_p)\to H^*(G;\F_p)$.  Moreover, it is a module homomorphism over the Steenrod algebra.
\end{thm}
We remark that the last sentence follows because transfer maps commute with the Steenrod operations~\cite{Evens}, as do maps induced on cohomology by group homomorphisms.

\begin{proof}
Let $A$ be the elementary abelian $p$-subgroup consisting of regular unipotent elements, whose existence is guaranteed by Corollary~\ref{ptorus}; we may assume $A\leq U$.  Fix an $s\in S$, and consider the composition
\[ A\to U\to U/U_s. \]
The composition is injective, because no regular unipotent lies in $U_s$.  Hence it is an isomorphism because $U/U_s\iso X_s\iso\F_q\iso A$ as groups.  We shall use this fact below.

Recall that $U_s\normal B$.  Because $B$ is generated by $T$, $U_s$, and $X_s$, the quotient is
\[ B/U_s=X_s\semidirect T. \]
For the remainder of this proof, all cohomology is with $\F_p$ coefficients, which are suppressed from the notation.
Now consider the map on cohomology given by
\[ (H^*(X_s))^T=H^*(B/U_s)\to H^*(B)\xrightarrow{\operatorname{tr}}
H^*(G)\xrightarrow{i^*} H^*(A). \]
The first map is induced by the quotient homomorphism, the second is the transfer map, and the third restriction.  We wish to show that this composition is injective, from which the claims in the theorem follow.
The double coset formula expresses the composition $i^*\circ\operatorname{tr}$ as a sum indexed over $A\backslash G/B$.  Since $A$ is elementary abelian, all of the terms vanish except those corresponding to the fixed points of $A$ on $G/B$.  But as $A$ contains regular unipotents, $B$ is the only such fixed point.  Therefore the composition $i^*\circ\operatorname{tr}$ equals the restriction map $H^*(B)\to H^*(A)$.
Hence the above composition is
\[ H^*(B/U_s)\to H^*(B)\to H^*(A), \]
which in turn equals
\[ H^*(B/U_s)\to H^*(U/U_s)\to H^*(U)\to H^*(A). \]
But the first map is injective since $p$ does not divide the index, while the composition of the other two is an isomorphism, as remarked at the start of this proof.  Hence the map in question is injective as desired.
\end{proof}

\begin{cor}\label{nonvanishing}
Let $G$ be a finite group of Lie type over $\F_q$ having Coxeter number at most $p$, which is nontrivial (of positive rank).  Then $H^*(G;\F_p)$ contains $H^*(GL_2(\F_q);\F_p)$, as a (graded) submodule over the Steenrod algebra. In particular,
\[ H^{r(2p-3)}(G;\F_p)\neq0. \]
\end{cor}
\begin{proof}
Pick any $s\in S$.  By the theorem, it suffices to express
$H^*(GL_2(\F_q);\F_p) = (H^*(\F_q;\F_p))^{\F_q^\times}$ as such a submodule of $(H^*(X_s;\F_p))^T$.  It is, since $X_s\iso\F_q$ with $T$ acting $\F_q$-linearly.
\end{proof}

In one case, namely the simply-connected groups in type $C$, there is a choice of $s$ for which
\[ \im(T\to\aut(X_s))<\F_q^\times \]
is a proper subgroup; that is, $T$ does not act transitively on the nonzero elements of $X_s$.  Here we get a stronger result than that of Corollary~\ref{nonvanishing}.
\begin{cor}
Suppose $2n\leq p$.  Then $H^*(Sp(2n,\F_q);\F_p)$ contains $H^*(SL_2(\F_q);\F_p)$, as a (graded) submodule over the Steenrod algebra. In particular,
\[ H^{r(p-2)}(Sp(2n,\F_q);\F_p)\neq0. \]
\end{cor}
\begin{proof}
We may assume $p$ is odd.  Let $\alpha_s$ be the long simple root.  Then the action of $T$ on $X_s\iso\F_q$ factors through
$2\cdot\F_q^\times$, the subgroup of squares.  (This may be computed directly; or, see the next section.)  The invariants of $H^*(\F_q;\F_p)$ by this group are $H^*(SL_2(\F_q);\F_p)$, and are nonvanishing\footnote{In the notation of~\cite[Lem.~A.1]{FP}, the generator is\newline
$a_{1,2}^0\wedge\cdots\wedge~a_{1,2}^{r-1}\otimes~(b_{1,2}^0)^{(p-3)/2}\cdots(b_{1,2}^{r-1})^{(p-3)/2}$. }
in degree $r(p-2)$.
\end{proof}
In the next section, we show that the groups $Sp(2n,\F_q)$ are essentially the only case in which this happens; that is, in all other cases the action map $T\to\F_q^\times$ on every root subgroup is surjective.

\section{Root action surjectivity}
Let $\bar G$ be a connected reductive group and $F:\bar G\to \bar G$ a $q$-Frobenius map.
Let $\bar T\leq\bar G$ be a maximal torus split over $\F_q$ and
\[ \Phi\subset\chi(\bar T)=\hom(\bar T,\bar\F_q^\times) \]
be its root system.  Let $\bar\alpha\in\Phi$.  Consider the induced map on finite groups
\[ \alpha:T\to\F_q^\times \]
of $T$ on $X_\alpha$.  In this section, we investigate the question of when the invariants $(H^*(X_\alpha;\F_p))^T$ of $T$ acting on a root subgroup can be bigger than the ``universal invariants''
$(H^*(\F_q;\F_p))^{\F_q^\times}$,
which would require that $\alpha$ fails to be surjective.

First of all, we consider $\bar T$ in isolation.  Let $\bar T$ be an algebraic torus split over $\F_q$.  The following lemma, easy to check in coordinates, shows when an arbitrary character fails to induce a surjective map on the finite groups.

\begin{lemma}
Let $n|(q-1)$ and $\bar\alpha\in\chi(\bar T)$.
Then \[ n|[\F_q^\times:\im(\alpha)] \]
if and only if $\bar\alpha$ is divisible\footnote{In general, we say that an element $x$ of an abelian group $A$ is divisible by $n\in\bbZ$ if there exists $y\in A$ with $x=n\cdot y$.  We say that $x$ is divisible if it is divisible by some $n>1$.} by $n$ in $\chi(\bar T)$.
\end{lemma}

Now return to the situation where $\bar T$ is a maximal $\F_q$-split torus in a connected reductive group $\bar G$, of which $\bar\alpha$ is a root.  The lemma says that $\alpha$ is surjective unless $\bar\alpha$ is divisible in $\chi(\bar T)$ by some integer dividing $q-1$.  Since no root may be divisible (by any integer greater than one) in the root lattice, this immediately proves:

\begin{cor}
If $\bar G$ has adjoint type (i.e.~$\chi(\bar T)=\bbZ\Phi$), then $\alpha$ is surjective for every root $\bar\alpha$.
\end{cor}

Furthermore, the lemma shows that surjectivity of $\alpha$ can never fail unless $\bar\alpha$ is divisible in the (larger) weight lattice.  But this weaker condition still never occurs except in one case:

\begin{lemma}
Let $\Phi$ be a (reduced, crystallographic) root system, and $\Lambda\geq\bbZ\Phi$ the associated weight lattice.  Then no root $\bar\alpha\in\Phi$ is divisible in $\Lambda$, unless (an irreducible component of) $\Phi$ has type $C_n$, in which case a long root $\bar\alpha$ is divisible by 2.
\end{lemma}
Indeed, for a (simple) root to be divisible by $m>1$ in $\Lambda$ is equivalent to the Cartan matrix for $\Phi$ containing a column which is divisible by $m$.  This only occurs in type $C_n$, in the case of the long simple root, and $m=2$~\cite[ch.~6]{Bourbaki}.

Since there are only two simple groups of Lie type $C_n$ (the simply-connected one and the adjoint one), the preceding three results show:

\begin{thm}
Let $G$ be a simple group of Lie type over $\F_q$ other than a symplectic group (i.e.~the simple simply-connected group of type $C_n$).  Then the action of $T$ on each root subgroup $X_\alpha$ of $G$ induces a surjective map $T\to\F_q^\times$.
\end{thm}

\section{Appendix: exponent of $U$}
We will show:
\begin{prop}\label{exponent}
In a finite group of Lie type over $\F_{p^r}$ with Coxeter number at most $p$, the Sylow $p$-subgroup $U$ has exponent $p$.
\end{prop}

Since $U$ is generated by the positive root subgroups, their height function yields a filtration of $U$, with length $h-1$:
\[ U_{\hgt\geq k} = \langle X_\alpha | \alpha\in\Phi^+, \hgt\alpha\geq k\rangle. \]
The root subgroups $X_\alpha$ are abelian and satisfy the ``Chevalley commutator formula''
\[ [X_\alpha,X_\beta]\leq\langle X_{i\alpha+j\beta} | i,j>0 \rangle \]
for all $\alpha,\beta\in\Phi$ with $\alpha\neq\pm\beta$.
(Note that $X_\gamma$ should be interpreted as trivial when $\gamma\notin\Phi$.)
It follows that the above filtration of $U$ is a central series.  In particular, the nilpotence class of $U$ is at most $h-1$.

Using this central series, we show that, when the Coxeter number is at most $p$, every nontrivial element of $U$ has order $p$.  This follows from Hall's ``commutator collection'' trick~\cite[Cor. 12.3.1]{Hall}:

\begin{prop}[Hall]
Let $P$ be a $p$-group of nilpotence class less than $p$.  Then for $a_i\in P$, we have
\[ (a_1\cdots a_r)^p = a_1^pa_2^p\cdots a_r^pS_1^p\cdots S_\ell^p, \]
where $S_i\in[P,P]$.
\end{prop}

\begin{cor}\label{expcor}
Let $P$ have a central series $1=P_0\leq\cdots\leq P_k=P$ of length $k<p$.  Suppose each $P_i$ is generated by elements of order $p$.  Then $P$ has exponent $p$.
\end{cor}

\begin{proof}
We use induction on $k$.  All elements of $P$ can be written as a product $a_1\cdots a_r$ with $a_i^p=1$.  Applying the proposition, we see that it suffices to show that $S^p=1$ for all $S\in[P,P]\leq P_{k-1}$.  This is true because $P_{k-1}$ satisfies the induction hypothesis.
\end{proof}

\begin{proof}[Proof of Proposition~\ref{exponent}]
Each level of the height filtration for $U$ is generated by root subgroups, each of which has exponent $p$.  Also the maximum height is $h-1\leq p-1$.  Hence the conditions of Corollary~\ref{expcor} are satisfied, showing that $U$ has exponent $p$.
\end{proof}

\end{document}